\documentclass[12pt, notitlepage]{article}
\usepackage[utf8]{inputenc}
\usepackage[english]{babel}
\usepackage{amsfonts}
\usepackage{amssymb}
\usepackage{amsmath}
\usepackage{amsthm}
\usepackage{tocloft}
\usepackage{CJK}
\usepackage[english]{babel}
\usepackage{chngcntr}
\counterwithin{equation}{section}

\makeatletter\@addtoreset{chapter}{part}\makeatother

\newcommand{\xdownarrow}[1]{%
  {\left\downarrow\vbox to #1{}\right.\kern-\nulldelimiterspace}
}

\begin{document}

\title{Leveled Sub-cohomology}

 \author{B. Wang\\
\begin{CJK}{UTF8}{gbsn}
(汪      镔)
\end{CJK}}

\date {}

\maketitle
 
\begin{abstract}
In this paper we define a functor-- leveled sub-cohomology.
(It bears no relation with the level of elliptic curves). 
It is based on  leveled cycles on a smooth projective variety, and 
 will be expected to reveal a  structure in the level.

\end{abstract}

\newcommand{\hookuparrow}{\mathrel{\rotatebox[origin=c]{90}{$\hookrightarrow$}}}
\newcommand{\hookdownarrow}{\mathrel{\rotatebox[origin=c]{-90}{$\hookrightarrow$}}}
\newcommand\ddaaux{\rotatebox[origin=c]{-90}{\scalebox{0.70}{$\dashrightarrow$}}} 
\newcommand\dashdownarrow{\mathrel{\text{\ddaaux}}}

\newtheorem{thm}{Theorem}[section]

\newtheorem{ass}[thm]{\bf {Claim} }
\newtheorem{prop}[thm]{\bf {Property } }
\newtheorem{prodef}[thm]{\bf {Proposition and Definition } }
\newtheorem{construction}[thm]{\bf {Main construction } }
\newtheorem{assumption}[thm]{\bf {Assumption} }
\newtheorem{proposition}[thm]{\bf {Proposition} }
\newtheorem{theorem}[thm]{\bf {Theorem} }
\newtheorem{apd}[thm]{\bf {Algebraic Poincar\'e duality} }
\newtheorem{cond}[thm]{\bf {Condition} }
\newtheorem{ex}[thm]{\bf Example}
\newtheorem{corollary}[thm]{\bf Corollary}
\newtheorem{definition}[thm]{\bf Definition}
\newtheorem{lemma}[thm]{\bf Lemma}
\newtheorem{con}[thm]{C}
\newtheorem{conj}[thm]{\bf Conjecture}

\bigskip

\bigskip

\bigskip

\maketitle

\bigskip

\tableofcontents

\section{Introduction}

Let $X$ be a smooth projective variety over the complex numbers.
The total Betti cohomology group $$H(X;\mathbb Q)=\sum_i H^i(X;\mathbb Q)$$
over $\mathbb Q$  is a $\mathbb Q$ linear space.
There are many subgroups such as the convineau filtration  ([5])
$$N^p H^q (X)\subset H^q(X;\mathbb Q),$$ 
sub-Hodge structures ([2])
$$ L^p H^q(X)\subset  H^q(X;\mathbb Q),$$ and Hodge filtrations 
 
$$F^pH^q(X)\subset H^q(X;\mathbb C)$$
over $\mathbb C$, 
 etc. 
They are all functorial on the category $SmProj/\mathbb C$ of smooth projective varieties over $\mathbb C$. 
 In this paper we are going to re-group them, so that
a symmetry induced by the Poincar\'e duality will emerge. At the meantime they become 
functorial not only on the category $Corr^0(\mathbb C)$ which includes $SmProj/\mathbb C$, but  also on a further category.

 \bigskip

More precisely we are going to axiomatize a sub-cohomology 
$$\mathcal H_k(X)$$ of $H(X;\mathbb Q)$. 
Then for  two leveled sub-cohomologies $\mathcal H_k(X), \mathcal J_{k'}(X)$, we give a sufficient  condition for
the intersection number  pairing  between  them to be non-degenerate. The non-degeneracy will imply a duality between them.
This is the symmetry  mentioned above. It is called the algebraic Poincar\'e duality, abbreviated as APD.  The primary targets are 
two non-trivial examples. They are \par
 (1) algebraically leveled  filtration  $\mathcal N_k(X)$ of totoal cohomology $H(X)$ at \par \hspace{1cc} level $k$, \par
 (2) Hodge leveled filtration  $\mathcal M_k(X)$  of totoal cohomology $H(X)$ at level  \par \hspace{1cc} $k$.
\par
 
They are flitrations over $\mathbb Q$ for the total cohomology $H(X)$. 
Briefly  $\mathcal N_k(X)$ is defined to be the linear span of 
all cohomology classes $\alpha\in H(X;\mathbb Q)$  supported on  an algebraic set of dimension at most
$[{k+dim_{\mathbb R}(\alpha)\over 2}]$, and $\mathcal M_k(X)$ is defined to be the linear span 
of $\mathbb Q$-subsapces of all sub-Hodge structures of level $k$. 
It is known that they form two ascending filtrations on 
$H(X;\mathbb Q)$
\begin{equation}
\mathcal N_0(X)\subset \mathcal N_1(X)\subset \cdots\subset 
H(X;\mathbb Q).\end{equation}

\begin{equation}
\mathcal M_0(X)\subset \mathcal M_1(X)\subset \cdots\subset 
H(X;\mathbb Q)\end{equation}
and 
$$\mathcal N_k(X)\subset \mathcal M_k(X).$$

In this paper we initiate a study of a duality among the leveled sub-cohomology which
include\par
 (a) APD1, a self duality within $\mathcal N_k$, \par
(b) APD2, a duality between $\mathcal N_k$ and $\mathcal M_k$, \par
(c) APD3, a self duality within $\mathcal M_k$. \par

  \bigskip

\bigskip

\section{Functor of leveled sub-cohomology}

\bigskip

\begin{definition} (Double functor)
Let $\mathcal W$ be a category and $\mathbf A$ be another category.
Let \begin{equation}\begin{array}{ccc}
\eta: \mathcal W&\rightarrow & \mathbf A
\end{array}\end{equation}
be a map equipped with two functors,  covariant $\eta_1$ and contravariant $\eta_2$.
We call $\eta$  a double functor.

\end{definition}
\bigskip

For the convenience, without a further explanation, we  use $X$ to denote a smooth projective variety of dimension $n$ over $\mathbb C$.
\bigskip

\begin{definition}
Let $Corr(\mathbb C)$ be a category, 
\par
(a) whose objects are smooth projective varieties over $\mathbb C$, \par
(b) whose morphisms from $X\to Y$ are rational correspondences $$\langle Z\rangle \in CH(X\times Y;\mathbb Q)$$\par

(c) whose compositions are the compositions of correspondences.  
\end{definition}

It is easy to check the graph of identity map is the identity of the category and the associativity of correspondence is the
associativity of the morphism.  This is not the $Corr^0(\mathbb C)$ from the Chow motives, nor $Cor(\mathbb C)$ of finite 
correspondences ([6]). 
\bigskip

\begin{definition}  Let $H(\cdot; \mathbb Q)$ the Betti cohomology of a smooth variety over $\mathbb C$.
We define a double functor, also denoted by  $H(\cdot; \mathbb Q)$ on $Corr(\mathbb C)$ as follows.\par

(a)
\begin{equation}\begin{array}{ccc}
Corr(\mathbb C) & \rightarrow & Linear\ spaces/ \mathbb Q.\\
X &\rightarrow & H(X; \mathbb Q)
\end{array}
\end{equation}

\par
(b) For any morphism $\langle Z\rangle \in CH(X\times Y;\mathbb Q)$ where $Z$ is an algebraic cycle
in $X\times Y$,  
we let  $P_X, P_Y$ be the projections from $X\times Y$ to $X$, $Y$ respectively. 
Then there is a morphism, 
\begin{equation}\begin{array}{ccc}
H(Y;\mathbb Q) & \rightarrow H(X;\mathbb Q) 
\end{array}\end{equation}
defined by
$$\langle Z\rangle ^\ast (\alpha)= (P_X)_\ast ((1\otimes \alpha)\cup \langle Z\rangle )$$
where $(P_X)_\ast$ is the integration along the fibre (because $P_X$ is a flat morphism). Notice
$(P_X)_\ast$ coincides with the Gysin homomorphism $(P_X)_!$ induced by $P_X$. This is the contravariant functor on  $H(X; \mathbb Q)$.
Similarly we define another morphism
\begin{equation}\begin{array}{ccc}
H(X;\mathbb Q) & \rightarrow H(Y;\mathbb Q) 
\end{array}\end{equation}
 by
$$\langle Z\rangle_\ast (\alpha)= (P_Y)_\ast ((\alpha\otimes 1)\cup \langle Z\rangle ).$$
This is the covariant functor.  Thus the cohomology $H(\cdot;\mathbb Q)$ is a double functor.
These two functors on the cohomology usually are not inverse to each other. They operate on different degrees.

\end{definition}
\bigskip

{\bf Remark} Double functor here is the union of two functors on the same object. 
The push-forward $\langle Z\rangle_\ast$ is just the pull-back with the transpose, 
$(\langle  Z\rangle^t)^\ast$.  
\bigskip

\bigskip

Notice the cohomology $H(\cdot; \mathbb Q)$ is commonly known as a contravariant functor on the different category $SmProj/\mathbb C$,
the smooth projective varieties over $\mathbb C$, 
\begin{equation}\begin{array}{ccc}
SmProj/\mathbb C & \rightarrow & Linear\ spaces/ \mathbb Q.\\
X &\rightarrow & H(X; \mathbb Q)
\end{array}
\end{equation}
If coupled with Gysin homomorphism,  it is also a double functor. 

\bigskip

 In the following we  define a sub-functor of the cohomology $H(\cdot;\mathbb Q)$. 

\begin{equation}\begin{array}{ccc}
Corr(\mathbb C) & \rightarrow & Linear\ spaces/ \mathbb Q.\end{array}
\end{equation}
where $Corr(\mathbb C)$ is the category of correspendences.

\bigskip

\begin{definition} 
Let $k$ be a whole number. 
A double functor
\begin{equation}\begin{array}{ccc}
\mathcal H_k(\cdot): Corr(\mathbb C) &\rightarrow &   Linear\ spaces/\mathbb Q
\end{array}\end{equation}
is called  a sub-cohomology leveled at $k$   if it satisfies \par

(1) \begin{equation} \mathcal H_k(\cdot)\subset H(\cdot; \mathbb Q).
\end{equation}
and it is a sub-double functor of $H(\cdot; \mathbb Q).$
\smallskip

(2)  For $X$ with $n< k$ where $n=dim_{\mathbb C}(X)$, 
\begin{equation}
\mathcal H_k(X)=0.
\end{equation}

For $X$ with $n\geq k$, 

\begin{equation}
\mathcal H_k(X)\subset \sum_{r=0}^{r=n-k} H^{2r+k}(X; \mathbb Q).
\end{equation}

\smallskip

(3) For each $X$,   \begin{equation}
\mathcal H_k(X)\cap \sum_{r\in [0, k]\cup  [2n-k, 2n]} H^r(X;\mathbb Q)
=\sum_{r\in [0, k]\cup  [2n-k, 2n]} H^r(X;\mathbb Q).
\end{equation}

\smallskip

(4) For $X, Y$ in $Corr(\mathbb C)$, 
$$\mathcal H_{k}(X)\otimes _\mathbb Q\mathcal H_{k'}(Y)\subset \mathcal H_{k+k'}(X\times Y).$$

\bigskip

The morphisms are the restrictions of the double functor on $H(\cdot; \mathbb Q)$.

A cohomology class in $\mathcal H_k(X)$, or its representative will be called an $\mathcal H_k$ leveled cycle (or class). \smallskip

\end{definition}

\bigskip

{\bf Remark}
The word ``level" is due to the condition (3).   An equivalent notion is the coniveau. 
However the coniveau will not reaveal  the  duality, called algebraic Poincar\'e duality  defined below.

\bigskip

For the convenience, we  let $u$ be a map from the $SmProj/\mathbb C$ to
$H^2(\cdot ;\mathbb Q)$  satisfying that 
$u(X)$ is a line in  $H^2(;\mathbb Q)$ generated by a very ample divisor. 
Use $u^i$ for the linear map
\begin{equation}\begin{array}{ccc}
H^{\bullet}(X;\mathbb Q) & \rightarrow & H^{\bullet+2i}(X; \mathbb Q)\\
\alpha &\rightarrow & \alpha \cup u^i.\end{array}
\end{equation}

In the context, we use the same notation $u^i$ to denote its restrictions. Use
$V$ to denote the generic hyperplane section that represents the class $u$. 
However $u$ is not a functor. \bigskip

\begin{definition}
Let $\mathcal H_k$ be a leveled sub-cohomology. 

For any $X\in Corr(\mathbb C)$,  primitive leveled sub-cohomology is defined to be

\begin{equation}\begin{array}{c}
\mathcal H_{k, prim}(X)=\mathcal H_k(X)\cap \biggl( \sum_{p\leq n} H^p_{prim}(X;\mathbb Q)+
\sum_{p>n} u^{2p-n}H^p_{prim}(X;\mathbb Q)\biggr). \end{array}
\end{equation}

We'll denote $$\sum_{p\leq n}H^p_{prim}(X;\mathbb Q)+\sum_{p>n} u^{p-n}H^{2n-p}_{prim}(X;\mathbb Q)$$
by $$H_{prim}(X;\mathbb Q).$$
(Both are not functors of $Corr(\mathbb C)$). 
\end{definition}

{\bf Remark} 
Notice  cycles in $\mathcal H_{k, prim}(X)$ for $p> {n}$ are not the conventional primitive cycles.

\bigskip

\begin{definition} Algebraic Poincar\'e duality (APD)
\par
(a) Let $\mathcal H_k, \mathcal J_k$ be two leveled sub-cohomology functors. 
For each $X$, 
if the intersection pairing on
\begin{equation}
\mathcal H_k(X)\times  \mathcal J_k(X).
\end{equation}
is a perfect pairing. We say the algebraic Poincar\'e duality, abbreviated as APD,  holds on these two leveled sub-cohomology functors.
By the Poincar\'e duality this pairing has to be between

\begin{equation}
(\mathcal H_k(X)\cap H^i(X;\mathbb Q))\times  (\mathcal J_k(X)\cap H^{2n-i}(X;\mathbb Q)).
\end{equation}

(b) If the intersection pairing on
\begin{equation}
\mathcal H_{k, prim}(X)\times  \mathcal J_{k, prim}(X).
\end{equation}
is a perfect pairing, we say the primitive APD on $\mathcal H_k, \mathcal J_k$ holds.

\end{definition}

\section{Convineau Filtration--the first example}\par
Algebraically leveled  filtration is a filtration re-grouped from the coniveau filtration. While we review 
the well-known defintions, we'll give another 
description using currents.  Recall that in [4], Grothendieck created a filtration $Filt'{^p}$, 
 called ``Arithmetic filtration, as it embodies
deep arithmetic properties of the scheme ". 
   This later was referred to as the coniveau filtration.
$$N^pH^{2p+k}(X)=\mathcal N_k(X)\cap H^{2p+k}(X;\mathbb Q).$$ 

It is defined as a linear span of kernels of the linear maps
\begin{equation}\begin{array}{ccc}
H^{2p+k}(X;\mathbb Q) &\rightarrow &H^{2p+k}(X- W;\mathbb Q)
\end{array} \end{equation}
for a subvariety  $W$ of codimension at least $p$. 
This is the cohomological view. 
In the same paper, Grothendieck  immediately interpreted it 
as a linear span of images of Gysin homomorphisms
\begin{equation}\begin{array}{ccc}
H^{dim(W)+2p+k-2n}(\tilde W;\mathbb Q) &\rightarrow &H^{2p+k}(X;\mathbb Q)
\end{array} \end{equation}
for a subvariety  $W$ of codimension at least $p$ with a smooth resolution $\tilde W$.  
 This is a view of mixed Hodge  structures ([1]). 
  We'll use another interpretation of the coniveau filtration. It is  through currents, which are known to unite
both homology and cohomology.  Let $\mathcal D'(X)$ be the space of currents over $\mathbb R$ on $X$. 
 Let $C\mathcal D'(X)$ be its subset of closed currents and $E\mathcal D'(X)$ be its subset of exact currents.
Then 
\begin{equation} {C\mathcal D'(X)\over   E\mathcal D'(X)}= \sum_i H^{i} (X;\mathbb C).\end{equation}
There is a restriction map on currents
\begin{equation}\begin{array}{ccc}
\mathcal R: \mathcal D'(X) &\rightarrow \mathcal D'(X-W)
\end{array}\end{equation}
for a subvariety $W$.

Using the formulas (3.3) and (3.4),  we define  $$\mathcal D^{p} H^{2p+k}(X)$$
to be the linear span of classes in $H^{2p+k} (X;\mathbb Q)$ such that they lie in
\begin{equation}
{C\mathcal D'(X)\cap kernel(\mathcal R)\over   E\mathcal D'(X)\cap kernel(\mathcal R)}.
\end{equation}
 for some $W$ of codimension at least $p$.
 \bigskip

We have the following description of the convineau filtration.
\bigskip

\begin{proposition}
Let $ X$ be a smooth projective variety over $\mathbb C$. 
Then 
\begin{equation}
\mathcal D^p H^{2p+k}(X)=N^p H^{2p+k}(X).
\end{equation} It says that 
the cohomology class $\alpha$ lies in
\begin{equation}
 N^p H^{2p+k}(X) 
\end{equation}
if and only if it is represented by a current whose support is contained in an algebraic set of codimension at least $p$.

\end{proposition}
 
\bigskip

\begin{proof}
By the definition \begin{equation}
\mathcal D^p H^{2p+k}(X)\subset N^p H^{2p+k}(X).
\end{equation}
Let's see the converse.  

If $\alpha\in N^p H^{2p+k}(X)$, by Cor. 8.2.8, [1], $\alpha$ is the Gysin image 
\begin{equation}\begin{array}{ccc}
H^{dim(A)+2p+k-2n}(\tilde W;\mathbb Q) &\rightarrow &H^{2p+k}(X;\mathbb Q)
\end{array} \end{equation}
for some algebraic subvariety $W$ of codimension at least $p$.  By the definition of the Gysin homomorphism
there is a singular cycle $\sigma$ in $\tilde W$ such that
the image of $\sigma$ under the map
$$\begin{array}{ccc}
\rho: \tilde W &\rightarrow X\end{array}$$
is Poincar\'e dual to $\alpha$. Since the support of the current $\rho_\ast ([\sigma])$ is in $W$, the cohomology class
satisfies 
$$\rho_\ast (\langle \sigma\rangle )\in kernel(\mathcal R).$$ Thus the current

$$\rho_\ast ([\sigma])$$ is reduced to 
an element of
$$\mathcal D^p H^{2p+k}(X).$$
This completes the proof.

\end{proof}

\bigskip

\section{Maximal sub-Hodge structure--the second example}

\begin{definition}
Let $\Lambda\subset H^{2p+k}(X;\mathbb Q)$ be a sub-group.
If $$\Lambda_{\mathbb C}=\Lambda^{p, p+k}\oplus \Lambda^{p+1, p+k-1}\oplus\cdots\oplus \Lambda^{p+k, p}$$
where $\Lambda^{i, j}$ are subspaces of $H^{i, j}(X;\mathbb C)$.  Then $\Lambda_{\mathbb Q}$ is said to be a sub-Hodge structure
of the Hodge structure on $H^{2p+k}(X;\mathbb Q)$.
Let $$M^pH^{2p+k}(X)$$ be the linear span of  subspaces $\Lambda_{\mathbb Q}$ for all sub-Hodge structures
$$\Lambda_{\mathbb Q}\subset  H^{2p+k}(X;\mathbb Q).$$
The index $p$ is called the coniveau, and $k$ is called the level.
\end{definition}
Above corollary of Deligne shows
$$N^pH^{2p+k}(X)\subset M^pH^{2p+k}(X).$$

\bigskip

\begin{proposition} Let $X, Y$ be two smooth projective varieties over $\mathbb C$. 
Let $Z$ be an algebraic cycle in $X\times Y$ of a pure dimension,  and $\langle Z\rangle \in CH(X\times Y)$ be it class
in the Chow group.
Then  $\langle Z\rangle _\ast$  and   $\langle Z\rangle ^\ast$  on the cohomology will preserve the level.
\end{proposition}

\begin{proof}
The pull-back and push-forward operation on cohomology induced from the correspondence $\langle Z\rangle$  , are 
 morphisms of Hodge structures.  As it known that the difference between 
$i, j$ for any $(i, j)$ type of cohomology class will be preserved under any morphism
of Hodge structures,  the level $k$ is defined to be the maximal difference of $i, j$ for all classes in the sub-Hodge structures. Thus it must be preserved under $Z$.

\end{proof}

\bigskip

\section{Examples of leveled sub-cohomology}

\bigskip

First we make a general claim. Let $Z\subset X$ be an embedding of
a smooth variety $Z$  into another smooth variety $X$ over $\mathbb C$. 
Let $K$ be a smooth subvariety of $X$ such that $K,  Z$ intersect transversally at 
a smooth subvariety $W$.  Let $\omega_{x\subset y}$ denote the cohomology Poincar\'e dual
to the submanifold $x$ in manifold $y$.  Let
 $j: Z\hookrightarrow X$ be  the inclusion map. 
\bigskip

\begin{lemma} Then
\begin{equation}
j^\ast(\omega_{Z\subset X})=\omega_{W\subset Z}.
\end{equation}
\end{lemma}
\bigskip

\bigskip

\begin{proof}
Because the intersection $W=K\cap Z$ is transversal.  The normal bundles satisfying
\begin{equation}
N_{W/Z}\subset N_{Z/X}.
\end{equation}
Furthermore the following diagram commutes
\begin{equation}\begin{array}{ccc}
N_{W/Z}&\stackrel{\psi}\hookrightarrow & N_{Z/X}\\
\downarrow && \downarrow\\
W &\hookrightarrow & Z.
\end{array}\end{equation}
Let $\eta_Z, \eta_W$ be the Thom classes of 
bundles $N_{Z/X}\to Z$ and $ N_{W/Z}\to W$.
Then \begin{equation} 
\psi^\ast(\eta_Z)=\eta_W.\end{equation}
Let's embed the formula (5.3) into the tubular neighborhoods of $W\subset Z$ and $Z\subset X$. Then formula (5.4) becomes 
(5.1). This completes the proof.

\end{proof}

\begin{proposition} \quad\par
Let $X, Y\in Corr(\mathbb C)$. Let $V$ be a hyperplane of 
the projective space containing $X$, and $u$ be its Poincar\'e dual.  

(1) The map $R$ 
$$ H(X;\mathbb Q)\to H(X\cap V; \mathbb Q)$$
 induced by the inclusion
$$X\cap V\hookrightarrow X$$ satisfies
\begin{equation}
 R (\mathcal H_k(X)) \subset \mathcal H_k(X\cap V).
\end{equation}
\smallskip

(2) \begin{equation}
u(X)\cup  \mathcal H_k(X) \subset  \mathcal H_k(X) .
\end{equation}
\smallskip

(3)  Let $Y\stackrel{i}\rightarrow  X$ be  a regular map. Denote the Gysin homormorphism
by $i _!$. 
 Then
\par

\begin{equation}
i _! (\mathcal H_k(Y))\subset \mathcal H_k(X).
\end{equation}
and
 $$i^\ast (\mathcal H_k(Y))\subset \mathcal H_k(X).$$
\bigskip

\end{proposition}
\bigskip

\begin{proof}
It suffices to show all these maps are realized by correspondences.\medskip

(1)  
Let $\Delta_{V, X}$ be the subvariety in $(X\cap V)\times X$,
$$\Delta_{V, X}=\{ (x, x): x\in X\cap V\}.$$
Then we have
\begin{equation}
\begin{array}{ccc}
\Delta_{V, X} &\subset  & (X\cap V)\times X\\
\downarrow & & \downarrow\scriptstyle {j} \\
\Delta_X & \subset & X\times X.
\end{array}\end{equation}
Then we use lemma 5.1 to obtain that
\begin{equation}
j^\ast (\omega_{ \Delta_X\subset (X\times X)})=\omega_{\Delta_{V, X} \subset   (X\cap V)\times X}.\end{equation}
Since $$\omega_{ \Delta_X\subset (X\times X)}=\langle \Delta_X\rangle$$
$$\omega_{\Delta_{V, X} \subset   (X\cap V)\times X}= \langle \Delta_{V, X}\rangle.$$

\bigskip

(2) Let $\Delta_V$ be the diagonal in $X\times X$, 
$$\Delta_V=\{(x, x): x\in V\}.$$
Then 
\begin{equation}
\langle \Delta_V\rangle=\langle \Delta_V\rangle \cup u\cup u^\ast.\end{equation}
where $u^\ast$ is the dual.  Then we check

$$\langle \Delta_V\rangle ^\ast (\alpha) =u\cup (\alpha)$$
for any cohomology class $\alpha\in H(X)$. 
\bigskip

(3)  Let 
$$G_i\subset Y\times X$$ be the graph of the map $i$. 
The Gysin homomorphism in section 7, is Poinca\'e daul to the induced map
on the singular homology
\begin{equation}\begin{array}{ccc}
H_p(Y) &\stackrel{i_\ast}\rightarrow & H_p(X).
\end{array}\end{equation}
The homomorphism $i_\ast$ can be expressed as the map $i_\#$ on singular chains.
Next we use simplecial complexes of each space. Let $\mathcal S_Y$ be a triangulation of $Y$.
This naturally induces  a triangulation $\mathcal S_G$ of $G_i$. Then
 $i_\#$ is the  $$(P_Y)_\#  ( (c\times X)\cap \mathcal S)$$
where $c$  is a cycle in  $\mathcal S_Y$, and $X$ is a complex containing the images of all 
$\mathcal S_Y$.  Then we reduce them to homology to obtain that
\begin{equation}
i_\ast=(P_Y)_\ast ( \langle c\rangle \cap \langle   G_i\rangle ).
\end{equation}

Applying the Poincar\'e duality to (5.12), we complete the proof for the Gysin homomorphism.
To see the pullbaclk $i^\ast$, 
let $\omega_{G_i}$ be the Poincar\'e dual of $G_i$ in $Y\times X$. Then 
the pullback $i^\ast(\beta)$ of the cohomology $\beta\in H^{p}(X;\mathbb Q)$ is 
the same as 
\begin{equation}
(P_Y)_\ast (\omega_{G_i}\cdot (1\otimes \beta)).
\end{equation}
(This is an assertion for any differentiable map).

\end{proof} 

\bigskip

\begin{proposition}\par

We make a convention that 
 \begin{equation}
N^i H^{2i+k}(X) = \left\{ \begin{array}{ccc}
N^i H^{2i+k}(X) & \mbox{for}
& i\in [0, 2dim(X)-k] \\
 0 & \mbox{for} & 2i+k\not\in [0,2dim(X)] \\
H^{2i+k}(X;\mathbb Q) & \mbox{for} & 2i+k\in [0, k]\cup [2dim(X)-k, 2dim(X)]
\end{array}\right.\end{equation}

The following sum of coniveau filtration 

 \begin{equation}
\sum_{r=-\infty}^{+\infty} N^r H^{2r+k}.
\end{equation}
gives   a rise to a  leveled sub-cohomology at the level $k$. 
We'll name it as algebraically leveled  filtration and denote it by $\mathcal N_k$.
Notice that 
$$\mathcal N_k(X)\cap H^{2r+k}(X;\mathbb C)= N^r H^{2r+k}(X).$$
\end{proposition}

\bigskip

\begin{proof}

 Fix a whole number $k$. 
We consider the map
\begin{equation}\begin{array}{ccc}
Corr(\mathbb C) &\rightarrow & Linear\ spaces/\mathbb C\\
X &\rightarrow & \mathcal N_k(X)
. \end{array}\end{equation}
 \bigskip

\bigskip

The morphisms are the restrictions of the double functors. 
Next we show it is covariant.
Let $X, Y, W$ be three projective varieties over $\mathbb C$. 
Let $$Z_1\in CH(X\times Y), Z_2\in CH(Y\times W).$$
Then it is suffices to show the composition criterion,  
\begin{equation}
(Z_2\circ Z_1)_\ast
=
\langle Z_2\rangle _\ast\circ \langle Z_1\rangle _\ast
\end{equation}
where $Z_2\circ Z_1$ is the composition of the correspondences.

Let $\alpha\in H(X\times Y\times W)$ be a cohomology having a homogeneous
degree. 
It will be sufficient to show the intersection
\begin{equation}
(Z_2\circ Z_1)_\ast (\alpha)=\langle Z_2\rangle _\ast\circ \langle Z_1\rangle _\ast(\alpha).
\end{equation}

We consider the triple cohomological intersection in the variety
$$X\times Y\times W, $$
\begin{equation}
\beta=\langle Z_1\otimes Y\rangle \cup \langle X\times Z_2\rangle\cup (\alpha\otimes \langle \times X\times Y\rangle)
.\end{equation}
Next we use two compositions of the same projection  $P_{W}^{XYW}$, 
\begin{equation}
P_{W}^{XW}\circ P_{XW}^{XYW}, P_{W}^{YW}\circ P_{YW}^{XYW}
\end{equation}
where the superscript indicates the domain of the projection, and the subscript indicates the target of the projection.
Then using the projection formula, we obtain the left-hand side of (5.18)
is 
\begin{equation}
(P_{W}^{XW}\circ P_{XW}^{XYW})_\ast(\beta)=(P_{W}^{XYW})_\ast(\beta),
\end{equation}
the right-hand side of (5.18) is
\begin{equation}
(P_{W}^{YW}\circ P_{YW}^{XYW})_\ast(\beta)=(P_{W}^{XYW})_\ast(\beta).
\end{equation}
This proves (5.18). So the functor is covariant. Similarly its transpose is also covariant.
We conclude that 
$$X \to  \mathcal N_k(X)$$
is a double functor. 
Next we show both morphisms preserve the level.
Let's first consider the pull-back. 
Let $X, Y$ be any smooth projective varieties over $\mathbb C$. 
Let $Z$ be an algebraic cycle in $X\times Y$ of complex codimension $l$., and 
$$\langle Z\rangle \in CH (X\times Y)$$ be the class in the Chow group. 
Let $\alpha\in N^qH^{2q+k}(Y)$. Then by Deligne's corollary, there is a
subvariety $A\subset Y$ such that $\alpha$ is
the Gysin image of 

\begin{equation}\begin{array}{ccc}
H^{dim(A)+2p+k-2n}(\tilde A;\mathbb Q) &\rightarrow &H^{2p+k}(Y;\mathbb Q)
\end{array} \end{equation}
where $\tilde A$ is the smooth resolution of $A$.    By the definition of the Gysin homomorphism
there is a singular cycle $\sigma$ in $\tilde A$ such that
the image of $\sigma$ under the map
$$\begin{array}{ccc}
\rho: \tilde A &\rightarrow Y\end{array}$$
is Poincar\'e dual to $\alpha$.  We may assume the intersection
$$Z\cap (X\times A)$$ is proper.  Applying the definition in cohomology,  

$$\langle Z\rangle_\ast(\alpha)$$
 is zero outside of
$$P_X (Z\cap (X\times A)).$$
This shows $$\langle Z\rangle_\ast(\alpha)\in N^{p'}H^{q'}(X;\mathbb C).$$
Next we calculate the level $p'$. 
\par
 
Let $dim(Z)=l$, $dim(X)=m$. 
Notice  $\langle Z\rangle^\ast$ sends 
$$ \begin{array}{ccc} H^{2q+k}(Y;\mathbb C) &\rightarrow & H^{2q+k+2l-2m}(X;\mathbb C),\end{array}$$

Suppose $\alpha$ lies in  $N^{q} H^{2q+k}(Y)$.  It lies in an algebraic cycle $\sigma_a$ of complex dimension at most 
$$m-{q},$$
Choose a cycle $Z'$ that is rationally equivalent to $Z$ such that the intersection
of $$Z'\cap (X\times \sigma_a)$$ is proper.
Then the complex dimension of $Z'\cap (X\times \sigma_a)$ is at most $$m+n-{q}-l$$
The complex dimension of the algebraic set
$$ P_X\biggl( supp(Z')\cap (X\times supp(\sigma_a)\biggr)$$ is also at most 
$$m+n-{q}-l.$$
Since $\langle Z\rangle^\ast(\alpha)$ lies in 
 $$P_X( supp(Z')\cap (X\times supp(\sigma_a)),$$
$\langle Z\rangle^\ast$ sends $N^{q} H^{2q+k}(Y)$ to
$$N^{l+q-m} H^{2q+k+2l-2m}(X).$$

Thus the level is $k$. Since the other morphism is the transpose of the same correspondence, 
the proof for the push-forward is identical after the change of the order of  $X$ and $Y$.

Thus $\mathcal N_k$ is a double functor. 
The conditions (1), (2) and (4)  are obvious. 
Since any cycle lies in $X$, 
$$N^0 H^{k}(X)=H^k(X;\mathbb C).$$
Because of the hard Lefschetz theorem ( 7, \S 0, [3]), for $k<n$, any $k$ cycle lies in 
a plane section of codimension $k$. Hence
$$N^{n-k}H^{2n-k}(X)=H^{2n-k}(X; \mathbb C).$$\par
By the Poincar\'e duality  the condition (3) is proved.  So the functor is leveled. 
  This completes the proof.
\end{proof}

\bigskip

\begin{proposition}  Let's have the same convention that
 \begin{equation}
M^i H^{2i+k}(X) = \left\{ \begin{array}{ccc}
M^i H^{2i+k}(X) & \mbox{for}
& i\in [0, 2dim(X)-k] \\
 0 & \mbox{for} & 2i+k\not\in [0,2dim(X)] \\
H^{2i+k}(X;\mathbb Q) & \mbox{for} & 2i+k\in [0, k]\cup [2dim(X)-k, 2dim(X)]
\end{array}\right.\end{equation}

The sum of  maximal sub-Hodge structure  
\begin{equation}
\sum_{r=-\infty}^{+\infty}M^{r}H^{2r+k}(X).
\end{equation}
is
a leveled sub-cohomology at a fixed level $k$. We name it as   Hodge leveled  filtration  and denote this  functor by $\mathcal M_k$. 
Notice 
$$\mathcal M_k(X)\cap H^{2r+k}(X;\mathbb C)= M^r H^{2r+k}(X).$$
\end{proposition}\bigskip

\begin{proof}

 Note  we defined 
\begin{equation}
\mathcal M_k(X)=\sum_{r=0}^{n-k} M^{r}H^{2r+k}(X).
\end{equation}

Since the induced homomorphisms  are all morphisms of Hodge structures. 
it gives a double functor.  All conditions (1)-(4) in definition 2.4 follow.

\end{proof}

\bigskip

\bigskip

\begin{ex}\quad \par

The cohomology $H(\cdot;\mathbb C)$ itself is  a leveled sub-cohomology at all levels.
This is the trivial leveled sub-cohomology, whose APD is the rational Poincar\'e duality.

\end{ex}
\bigskip

\begin{ex}\quad \par

Let $CH_{alg}^p(X)$ be  the Chow group of algebraic cycles algebraically equivalent to zero.
Then there is an Abel-Jacobi map
\begin{equation}\begin{array}{ccc}
AJ: CH_{alg}^p(X) &\rightarrow & J^p(X)
\end{array}\end{equation}
Let $J_a^p$ be its image. Since it is a sub-torus, the tangent space 
$TJ_a^p$ is contained in $H^{p-1, p}(X;\mathbb C)$.
We let 
\begin{equation}
H_a^{2p-1}(X;\mathbb C)=TJ_a^p\oplus \overline{TJ_a^p}.
\end{equation}
It was proved in [8], $H_a^{2p-1}(X;\mathbb C)$ has a sub-Hodge structure. So  there is a subspace 
$$H_a^{2p-1}(X;\mathbb Q)\subset H^{2p-1}(X;\mathbb Q)$$ such that
$$H_a^{2p-1}(X;\mathbb C)\simeq H^{2p-1}_a(X;\mathbb Q)\otimes \mathbb C, $$
and 
$$H_a^{2p-1}(X;\mathbb C)$$ is 
called algebraic part of cohomology.  
It is known that  $$H^1(X;\mathbb Q)=H^1_a(X;\mathbb Q), H^{2n-1}(X;\mathbb Q)=H^{2n-1}_a(X;\mathbb Q).$$
Therefore furthermore according to the definition 2.4, the algebraic part of cohomology 
$$\sum_{i=odd}H_a^i(\cdot;\mathbb Q)$$  defined by Murre   is a
leveled sub-cohomology at level 1. \bigskip

Actually Murre went further and showed  that
\begin{equation} \sum_{i=odd} H_a^i(\cdot;\mathbb Q)=\mathcal N_1.
\end{equation}

\end{ex}\bigskip

\begin{ex}
\quad \par
 The image of cycle maps, $$A(\cdot)=\sum_i A^i(\cdot)$$ is a leveled sub-cohomology  $\mathcal N_0$ at level $0$.

\end{ex}
\bigskip

\begin{ex}
\quad \par
Primitive cohomology $H_{prim}$ is not  a leveled sub-cohomology. \par
Primitive leveled sub-cohomology is a not  leveled sub-cohomology. \par
So they are not functors. \bigskip

\end{ex}
\bigskip

\bigskip

\begin{ex}
Let $X$ be a smooth projective variety defined over $\mathbb C$. Let $\tau\in Gal(\mathbb C/\mathbb Q)$.
Then there is another smooth projective variety $X_{\tau}$ over $\mathbb C$ defined by ideal 
$\tau (I(X))$ where $I(X)$ is the ideal defining $X$. See [9] for detailed discussion.

Through the algebraic de Rham cohomology, we obtain the isomorphism
\begin{equation}\begin{array}{ccc}
\tau: H^\bullet (X;\mathbb C) &\rightarrow & H^\bullet (X_\tau;\mathbb C)
,\end{array}\end{equation}
where the $\tau$ is induced from the isomorphism of the algebraic de Rham cohomology.
So $\tau$ is an isomorphism of $\mathbb C$ linear spaces, but it is not an isomorphism of 
the $\mathbb Q$ linear spaces. 
We call the subgroup 
\begin{equation} A^{p}_\tau H^{2p+k}(X)\subset H^{2p+k}(X;\mathbb Q)
\end{equation}
 the relative leveled sub-cohomology, 
 if $A^{p}_\tau H^{2p+k}(X)$ is the maximal sub-space such that 
\begin{equation}
\tau \biggl(  A^{p}_\tau H^{2p+k}(X)\biggr)\subset M^pH^{2p+k}(X_\tau)\otimes \mathbb C.
\end{equation}
 
We defined  $$ A^{p}H^{2p+k}(X)$$ to be the absolute leveled sub-cohomology if (5.32) holds for all
$\tau\in Gal(\mathbb C/\mathbb Q)$. 

Let $Y$ be another smooth projective variety over $\mathbb C$. 
Let $Z$ be an algebraic correspondence between $X$ and $Y$.
Then $Z_\tau$ will defined an algebraic correspondence between $X_\tau, Y_\tau$.
Then we have two commutative diagrams
\begin{equation}\begin{array}{ccc}
H(X;\mathbb C) &\rightarrow &H(X_\tau;\mathbb C) \\
\downarrow\scriptstyle{\langle Z\rangle _\ast} && \downarrow\scriptstyle{\langle Z_\tau\rangle _\ast}\\
H(Y;\mathbb C) &\rightarrow & H(Y_\tau;\mathbb C)
\end{array}\end{equation}

\begin{equation}\begin{array}{ccc}
H(X;\mathbb C) &\rightarrow &H(X_\tau;\mathbb C) \\
\uparrow\scriptstyle{\langle Z\rangle^\ast} && \uparrow\scriptstyle{\langle Z_\tau\rangle ^\ast}\\
H(Y;\mathbb C) &\rightarrow & H(Y_\tau;\mathbb C).
\end{array}\end{equation}

These diagram imply that 
$$ A^{p}H^{2p+k}(X), A^{p}_\tau H^{2p+k}(X)$$ 
both form leveled sub-cohomology. Precisely if
we let
$\mathcal A_{\tau, k}$ be a double functor
with
\begin{equation}
\mathcal A_{\tau, k}(X)=\sum_{p=-\infty}^{+\infty} A^{p}_\tau H^{2p+k}(X).  
\end{equation}
(use a convetion as in $\mathcal N_k$)
and $\mathcal A_{ k}$ be a double functor
with
\begin{equation}
\mathcal A_k(X)=\sum_{p=-\infty}^{+\infty} A^{p} H^{2p+k}(X).  
\end{equation}
Then they both are $k$ leveled sub-cohomology.
\bigskip

 Since $\tau$ preserves the Hodge filtration, 
\begin{equation}
\mathcal N_k\subset \mathcal A_k\subset \mathcal M_k.
\end{equation}

But for arbitrary $\tau\in Gal(\mathbb C/\mathbb Q)$, there is only
\begin{equation}
\mathcal N_k\subset \mathcal A_{\tau, k}.
\end{equation}

Notice $\mathcal A_0$ consists of the absolute Hodge cycles.

\end{ex}

\bigskip

{\bf Remark}\par

In the examples, we have the relations 
$$\mathcal N_k\subset \mathcal M_k\subset H(\cdot;\mathbb Q)$$
and 
$$\mathcal N_k\subset \mathcal A_k\subset \mathcal A_{\tau, k}\subset H(\cdot;\mathbb Q).$$
Hodge conjecture leads to a question: is $\mathcal N_k$  the non trivial,  maximal leveled sub-cohomology?
 
\bigskip

\bigskip

\section{APD on leveled sub-cohomology}

\bigskip

\begin{theorem}
If the primitive APD   on a pair of 
 leveled sub-cohomology holds, then APD holds on the  pair .
\end{theorem}

\bigskip

Let's start with an easy lemma which may be well-known (see, for instance, (4.6), [10]).   

\bigskip

\begin{lemma}  Let $V$ and $Z$ be two smooth  projective varieties over $\mathbb C$. Let 
$$\begin{array}{ccc} i: Z &\rightarrow &  V\end{array}$$
be  the inclusion map. Let $\theta\in H^{\bullet}(V;\mathbb Q)$ be cohomology class. 
$\omega_Z\in H^{\bullet}(V;\mathbb Q)$ be the Poincar\'e dual of $Z$ in $V$. Then\par

(a)
\begin{equation}
i_!i^\ast(\theta) =\theta \cdot \omega_Z.
\end{equation}
\end{lemma}
(b) for any cohomology $\eta\in H^{\bullet}(V;\mathbb Q)$, the intersection numbers satisfy
\begin{equation}
(\eta, \omega_Z, \theta)_V=(i^\ast(\eta), i^\ast(\theta))_Z.
\end{equation}

\begin{proof}  (a) We use de Rham cohomology. 
Let's denote the de Rham representatives of $\theta$ and $\omega_Z$ by the same letter
$\theta$ and $\omega_Z$.
Let $\phi$ be a closed $C^\infty$ form on $V$.
Then it suffices to show 
\begin{equation}\int_V i_!i^\ast(\theta)\wedge \phi=\int_V\theta \omega_Z\wedge \phi.\end{equation}
Since both sides of (6.1) equals to
$$\int_Z\theta\wedge \phi$$
we complete the proof of (a). \par

(b) As above we use de Rham cohomology. Then according part (a)
left hand side of (6.2) is
$$\int_Z\eta \wedge \theta.$$
Using the intersection in de Rham cohomology, right hand side of (6.2)
is also
$$\int_Z\eta \wedge\theta.$$
This completes the proof of (b) 

\end{proof}

 \bigskip

\begin{definition} (Plane-sectional decomposition).
Let $u$ be a map  from the set of objects of $SmProj/\mathbb C$ to
$H^2(;\mathbb Q)$ satisfying
$u(X)$ is a line in  $H^2(;\mathbb Q)$ generated by a very ample divisor. 

Let $\mathcal H_k$ be a leveled sub-cohomology.
Let 
\begin{equation}
\mathcal H_k= L_0\oplus  L_1\cdots \oplus \cdots
\end{equation}
where 
$$
 L_i$$
is a direct sum complement of $ker(u^i|_{\mathcal H_k})$ in the $ker(u^{i+1}|_{\mathcal H_k})$,
$$L_i\oplus ker(u^i|_{\mathcal H_k})=ker(u^{i+1}|_{\mathcal H_k})$$
where the decomposition is not unique.  So
for each $X$, $\mathcal H_k(X)$ will be decomposed into finitely many $L_i(X)$.

\end{definition}

{\bf Remark} The decomposition is not unique. 

\bigskip

\bigskip

\begin{proof} of theorem 6.1: Let $p, k$ be two fixed whole numbers. Let $\mathcal H_k, \mathcal J_k$ be two leveled sub-cohomologies. 

We'll use the notations
 $$\begin{array}{c} \mathcal H_k^i=\mathcal H_k\cap H^i(\cdot; \mathbb Q)\\
\mathcal J_k^i=\mathcal J_k\cap H^i (\cdot; \mathbb Q).
\end{array}.$$

Then we apply induction on the dimension of $X$. When $dim(X)$ is the smallest for $p, k$, which is ${p+k\over 2}$. Then
both $$\mathcal H_k^p, \mathcal J_k^{2n-p}$$ are back to the usual cohomology $H^{\bullet}(X;\mathbb Q)$. By the rational Poincar\'e duality, the APD holds.  Next we  assume APD holds for $dim_{\mathbb C}(X)< n-1$. Consider the $X$ with $dim_{\mathbb C}(X)=n$. 
It suffices to prove that
\begin{equation}\begin{array}{ccc}
\mathcal H^{p}_k(X) & \stackrel{\mathcal P}\rightarrow & ( \mathcal J^{2n-p}_k(X))^\vee
\end{array}\end{equation}
is surjective,
where $\mathcal P$ is the map induced from the intersection form.

Next we consider two cases:
\par
(1) $p>n$.   Let's recall our goal:
for any given $\alpha\in H^p(X;\mathbb Q)$, we need to find 
$\mathcal H_k$ leveled $\alpha_{a}$ such that 
\begin{equation}
(\alpha_a, \omega)_X=(\alpha, \omega)_X.
\end{equation}
for all $\omega\in \mathcal J_k^{2n-p}(X) $.
In this statement we regard the intersection pairing $(\alpha, \bullet)_X$ as 
an element in $( \mathcal J^{2n-p}_k(X))^\vee$. 

By the hard Lefschetz theorem, the class
\begin{equation}
\alpha=u\cup \beta.
\end{equation}
where $\alpha\in H^{p}(X;\mathbb Q), \beta\in H^{p-2}(X;\mathbb Q)$.
Then applying lemma 6.2, we have the triple intersection number
$$(\beta, u, \omega)_X=(\beta_Y, \omega_Y)_Y,$$
where $Y$ is a smooth hyperplane section of $X$ and $(\bullet)_Y$ is the restriction of the cohomology to 
$Y$.   By the induction, since $\omega_Y$ is $\mathcal J_k$-leveled, there is an $\mathcal H_k$ leveled
cycle $\alpha_Y$, such that
\begin{equation}
(\beta_Y, \omega_Y)_Y=(\alpha_Y, \omega_Y)_Y.
\end{equation}
Let $i_!$ be the Gysin homomorphism 
from $$H^{\bullet}(Y;\mathbb Q)\to H^{\bullet+2}(X;\mathbb Q)$$
which maps cycles leveled at $k$ to  cycles at the same level. 
Then applying lemma 6.2 again, we obtain
\begin{equation}
\biggl (i_!(\alpha_{Y}),  \omega\biggr)_{ X}=(\alpha_Y, \omega_Y)_Y=(\alpha, \omega)_X.\end{equation}
Thus $i_!(\alpha_{Y})$ is the $\mathcal H_k$ leveled cycle we are looking for.

\smallskip 

(2) $p\leq n$. We  need to first decompose $(\mathcal J_k^{2n-p}(X))^\vee$.  This is originated from the
decomposition in definition 6.3, the plane-sectional decomposition,
\begin{equation}
\mathcal J_k^{2n-p}(X)= L_0(X)\oplus  L_1(X)\cdots \oplus  L_{[{p-1\over 2}]}(X).
\end{equation}

By the topological Poincar\'e duality we always have
the surjective map 
\begin{equation}\begin{array}{ccc}
\mathcal P: H^{p}(X;\mathbb Q) &\rightarrow & (L_i(X))^{\vee}
\\
\alpha &\rightarrow & \alpha\cap (\bullet)
\end{array}\end{equation}
for each  $\ 0\leq i\leq [{p-1\over 2}]$. 

 Due to the definition of $L_i(X)$,  the  map
\begin{equation}\begin{array}{ccc}
L_{i}(X) &\stackrel{u^{i}}\rightarrow & H^{2n-p+2i}(X;\mathbb Q).
\end{array}\end{equation}
is injective.
Therefore the dual map which is still denoted by $u^{i}$,

\begin{equation}\begin{array}{ccc}
  H^{p-2i}(X;\mathbb Q)&\stackrel{u^{i}}\rightarrow & (L_{i}(X))^\vee
\end{array}\end{equation}
is  surjective.  
\par

Hence
\begin{equation}\begin{array}{ccc}

\oplus_{i=1}^{[{p-1\over 2}]} H^{p-2i}(X;\mathbb Q) &\stackrel{\sum_i u^i}\rightarrow &
\oplus_{i=1}^{[{p-1\over 2}]} (L_i(X))^{\vee}
\end{array}\end{equation}
is also surjective. 

Let $Y$ be a smooth subvariety such that
\begin{equation}
[Y]=V^i\cap X,
\end{equation}
where $1\leq i<n$. Thus $Y$ is also an irreducible,  smooth projective variety.

Then for any $\omega\in L_{i}(X), i\neq 0$,
we consider the  triple intersection number
\begin{equation}
( \alpha_i,  u^i, \omega)_X
\end{equation} 

Using lemma 6.2, we obtain that

\begin{equation}
( \alpha_i,  u^i, \omega)_X= (\alpha_{i, Y}, \omega_{Y})_{ Y}
\end{equation}
where $(\cdot)_Y$ is the restriction of the cohomology to its submanifold $Y$.
Notice $\omega_Y$ is the pull-back of $\omega$ which must be $\mathcal J_k$-leveled and
$Y$ has dimension lower than $n$. By the induction, we
obtain a $\mathcal H_k$-leveled cycle $\alpha_{i, Y}^a$ 
in $Y$ such that, 
\begin{equation}
(\alpha_{i, Y}^a, \omega_Y)_Y=(\alpha_{i,Y}, \omega_Y)_Y.
\end{equation}
Let $i_!$ be the Gysin homomorphism 
from $$H^{\bullet}(Y;\mathbb Q)\to H^{\bullet+2i}(X;\mathbb Q)$$
which maps cycles leveled at $k$ to  cycles leveled at $k$. 
Then applying lemma 6.2 again, we obtain
\begin{equation}
\biggl (i_!(\alpha_{i, Y}^a),  \omega\biggr)_{ X}=(\alpha,  u^i, \omega)_{ X}=\psi^i (u^i\omega),
\end{equation}
where $i_!(\alpha_{i, Y}^a)$ is $\mathcal H_k$-leveled. 
This show the surjectivity of the map 
\begin{equation}\begin{array}{ccc}
\mathcal H^{p}_k(X) & \stackrel{\mathcal P}\rightarrow & ( \oplus_{i\neq 0} L_i(X))^\vee. 
\end{array}\end{equation}

 Now we work with $(L_0(X))^\vee$.  Let $\omega\in L_0(X)$ be the testing cycle. 
 As before we consider $\alpha_0\in H^{p}(X;\mathbb Q)$ that reprensets an element in 
$(L_0(X))^\vee$. 
  
For any such $\alpha_0$, there is the Lefschetz decomposition
\begin{equation}
\alpha_0=\alpha_0^0+\sum_{l\geq 1} u^l \alpha_0^l.
\end{equation}
Using the same inductive argument above, we obtain a $\mathcal H_k$ leveled cycle $\alpha_0(1)$ such that
\begin{equation}
(\alpha_0(1), \omega)_X=(\sum_{l\geq 1} (u^l\alpha_l), \omega)_X,
\end{equation}
for any $\omega\in L_0(X)$.  By Lefschetz decomposition 
\begin{equation}
\omega=u^{n-p} \omega_{p}+u^{n-p+1}\omega_{p-2}+\cdots +u^{n-p+[{p\over 2}]}\omega_{p-2 [{p\over 2}]}.
\end{equation}
where $\omega_j\in H^j_{prim}(X;\mathbb Q)$. 
Notice $\omega u^i=0$ for all $1\leq i\leq [{p-1\over 2}]$. Hence all primitve cycles
$$\omega_{p-2}=\omega_{p-3}=\cdots=\omega_{p-2 [{p\over 2}]}=0.$$
Therefore 
\begin{equation}
\omega=u^{n-p} \omega^{p}
\end{equation}
where $\omega^{p}$ is primitive.   By the assumption of primitive APD (notice $ \omega=u^{n-p} \omega^{p}$ is
 $\mathcal J_k$ leveled), 
there is a primitive, $\mathcal H_k$ leveled cycle $\alpha_0(2)$  such that 
\begin{equation}
(\alpha_0(2), \omega)_X=(\alpha_0^0, \omega)_X.
\end{equation}
Thus
\begin{equation}
(\alpha_0(1)+\alpha_0(2), \omega)_X=(\alpha_0, \omega)_X.
\end{equation}
Now we combine all components in the decomposition 
\begin{equation}
( \mathcal J^{2n-p}_k(X))^\vee= (L_0(X))^{\vee}\oplus \cdots \oplus (L_{[{p-1\over 2}]}(X))^{\vee}.
\end{equation}

For any element
$\psi\in ( \mathcal J^{2n-p}_k(X))^\vee$,  it is decomposed as
\begin{equation}
\sum_{i=0}^{[{p-1\over 2}]} \psi^i
\end{equation}
where $\psi^i\in (L_i(X))^\vee$ and
$\psi^i$ can be represented through intersection form by the cycles
$\alpha_i$. Then we can find the $\mathcal H_k$ leveled cycle
\begin{equation}
\alpha_0(1)+\alpha_0(2)+\sum_{i\neq 0} i_! (\alpha_{i, Y}^a)
\end{equation}
such that 
its Poincar\'e dual is $\psi$.  
We complete the proof.

\end{proof}

\bigskip

\section{Glossary}
(1) If $X\stackrel{i}\rightarrow Y$ is a continuous map between two real compact manifolds, \par \hspace{1cc} then
the induced homomorphism $i_!$ in the graph, 
\begin{equation}\begin{array}{ccc}
H_p(X;\mathbb Q) &\stackrel{i_\ast}\rightarrow & H_p(Y;\mathbb Q)\\
\\
\scriptstyle{Poincar\acute {e}}\Bigg\Updownarrow \scriptstyle{duality} & &\scriptstyle{Poincar\acute {e}}\Bigg\Updownarrow \scriptstyle{duality} \\\\
H^{dim(X)-p}(X;\mathbb Q) &\stackrel{i_!}\rightarrow & H^{dim(Y)-p}(Y;\mathbb Q)\end{array}\end{equation}
 \par \hspace{1cc} will be called Gysin homomorphism.

\par

(2) $\mathbb M^{p, 2p+k}(X)$ is the maximal sub-Hodge structures of coniveau $p$ at level \par \hspace{1cc} $k$.\par
(3)   $N^p H^{2p+k}(X)$ is the coniveau filtration of coniveau $p$ at level $k$.\par
(4) $Hdg^{^{\bullet}}(X)$ is the subspace  spanned by Hodge classes.\par
(5)  $A^{\bullet}(X)$ is the subspace of the rational cohomology, spanned by \par \hspace{1cc} algebraic cycles.\par
(6) $a ^\vee$ denotes the dual of a vector space if $a$ is a vector space or a \par\hspace{1cc} vector.\par
(7)  $a ^\ast$  denotes a pullback in various situation depending on the  \par\hspace{1cc} contex.   \par
(8) $a_\ast$ denotes a pushforwad in various situation depending on the  \par\hspace{1cc} contex.\par
(9) $\langle  a\rangle $ denotes a classes in various groups represented by an object $a$.\par
(10) $(\cdot, \cdot)_X$ is the intersection number in $X$ between a pair of the same \par\hspace{1cc} or/and different types   of  objects.\par
(11) $(\cdot, \cdot, \cdot, ....)_X$ denotes the intersection number among multiple objects. \par

(12) $CH$ denotes the Chow group, $CH_{alg}$ denotes the subgroup of cycles \par\hspace{1cc} algebraically equivalent to zero. \par
(13) $J$ denotes the intermediate Jacobians.\par

(14) We'll drop the name ``Betti" on the cohomology. So all cohomology \par\hspace{1cc} are Betti cohomology.\par

\bigskip

\end{document}